\documentclass{amsart}

\usepackage{palatino} 
\usepackage{pinlabel}
\usepackage{pdfpages}
\usepackage{hyperref, multicol}
\usepackage{tikz-cd}

\usepackage{amsmath,amssymb,amsthm,verbatim, amsfonts,amscd,flafter,epsf, epsfig,graphicx,verbatim,pinlabel,mathrsfs}
\usepackage[all]{xy}
\usepackage{epsf}
\usepackage[abs]{overpic}
\usepackage{epstopdf}
\usepackage{color}
\usepackage{mathtools}
\usepackage{amsthm}

\definecolor{red}{rgb}{1,0,0}
\definecolor{blue}{rgb}{0,0,1}
\definecolor{green}{rgb}{0,1,0}

\newtheorem{theorem}{Theorem}[section]
\newtheorem{lemma}[theorem]{Lemma}

\newtheorem{corollary}[theorem]{Corollary}
\newtheorem{proposition}[theorem]{Proposition}
\newtheorem{question}[theorem]{Question}

\theoremstyle{definition}
\newtheorem{definition}[theorem]{Definition}

\newtheorem{remark}[theorem]{Remark}

\newtheorem{example}[theorem]{Example}

\usepackage{ragged2e}

\def\H{\mathcal{H}}

\title{Morse diagrams, Murasugi sums, and the mapping class group}
\author[Jack Brand]{Jack Brand}

\author[David T. Gay]{David Gay}

\address{ Department of Mathematics, University of Georgia, Athens, GA 30602}
\email{dgay@uga.edu}

\author[Joan Licata]{Joan Licata}
\address{Mathematical Sciences Institute, Australian National University  \& International Research Lab, France-Australia Mathematical Sciences and Interactions}

\email{joan.licata@anu.edu.au}

\begin{document}

\maketitle

\begin{abstract}   A \textit{combinatorial Morse structure} encodes a mapping class for a surface with boundary, and the data may be efficiently represented via a \textit{Morse diagram}.  This  diagram  determines  an open book decomposition of a 3-manifold, and hence, a contact structure on that 3-manifold.   We examine how combinatorial Morse structures behave under the connect sum of open books, with particular attention paid to the case of negative stabilisation.  This leads to a diagrammatic criterion for detecting overtwisted contact structures.  Finally, in the case of open books with one-holed torus pages, we classify all the Morse diagrams associated to a fixed open book decomposition.
\end{abstract}

\section{Introduction}

Interplay between different branches of mathematics establishes some as givers and others as takers.  While  contact geometry takes techniques from symplectic geometry, PDE, and algebra, it also plays a significant role as a giver.  In dimension 3, contact geometry has been instrumental in proving important results about non-contact phenomena such as the Property P Conjecture and the detection of knot genus via Heegaard Floer homology  \cite{KM}, \cite{OSzgenus}.  This supports the perspective that the contact geometry of 3-manifolds is fundamental: it lurks in the background of any topological study, regardless of whether it's directly addressed.  The relationship between open book decompositions and contact structures is a wonderful example of this: fibered links in three-manifolds were studied for decades before  Thurston and Winkelnkemper proved that each open book corresponds to a contact structure \cite{TW}.  Suddenly, old theorems could be interpreted in a new light, while the contact perspective prompted new questions.  Today, one may choose between studying open book decompositions as topological objects or as contact geometric ones, and this paper moves fluidly between these two views.  

The objects in this paper were first introduced to address a contact geometric question: what is front projection in an arbitrary contact manifold?  The answer proposed in \cite{GL} encodes an open book decomposition (and hence, a contact structure) in a decorated surface called a \textit{Morse diagram}; the image of a Legendrian link in a Morse diagram behaves similarly to a front projection in $\mathbb{R}^3$.  However, the value of Morse diagrams extends beyond their applications to knot theory.  A Morse diagram efficiently describes a complicated object: an equivalence class of Morse structures on an open book decomposition of a contact $3$-manifold.  In many applications, it's preferable to discard some of this data and focus on a less specific object like the open book decomposition, the contact manifold, or even the underlying $3$-manifold.  This requires replacing the equivalence relation on Morse diagrams with weaker relations to capture the desired level of specificity.   For example, if the object of interest is  a fixed open book, we can look for relations that capture all the Morse structures compatible with that open book.  To study a contact structure, the Giroux Correspondence tells us to consider all open books related by positive stabilisation. If one is interested only in the topological 3-manifold, then the set of relevant moves expands to include negative stabilisation and Stallings twists.

The results in this paper represent a step towards such a fully flexible tool, as we describe how Morse diagrams change under open book stabilisation.  In the special case of open books with one-holed torus pages, we can say more, giving a complete set of moves that relate all the different Morse diagrams compatible with a fixed open book.  Each of these results may be seen as progress in the larger program to develop a completely combinatorial approach to studying 3-manifolds and their contact structures via combinatorial Morse diagrams.

\subsection{Results}
This paper takes  Morse diagrams as the primary objects of study and develops diagrammatic operations that reflect operations on the associated open book decompositions.  These may be understood at either the topological or contact geometric level, but support for the latter comes in the form of a new criterion for detecting overtwisted contact manifolds directly from their Morse diagrams.  

There is a well-known operation on $3$-manifolds with open book decompositions called the Murasugi sum.  Given two manifolds with open book decompositions $(\Sigma_{1} ,\phi^{1})$ and $(\Sigma_{2},\phi^{2})$, the Murasugi sum produces an open book decomposition $(\Sigma_{1} * \Sigma_2,\overline{\phi}^{1}\circ \overline{\phi}^{2})$ on the connected sum of two original manifolds. In Section~\ref{sec:Murasugi}, we define a splicing operation that takes two Morse diagrams $X^1$ and $X^2$ and produces a new Morse diagram $X^1 * X^2$. Our first theorem is the following:

\begin{theorem}\label{thm:splice}  The splice $X^1*X^2$ is a Morse diagram for the open book $(\Sigma_{1} * \Sigma_2,\overline{\phi}^{1}\circ \overline{\phi}^{2})$.
\end{theorem}

The existence result in \cite{GL} (Proposition 3.3)  implies already that given any pair of open books, the connect sum $(\Sigma_{1},\phi_{1})\#(\Sigma_{2},\phi_{2})$ admits a Morse structure, but this is not constructive.  Here, we show how Morse diagrams for  the factors give rise to a canonical Morse diagram for their connect sum. 

The simplest nontrivial example comes from taking a Murasugi sum with either the left- or right-handed Hopf band open book decomposition of $S^3$; this operation is called negative or positive stabilisation. Negative stabilisation produces an open book supporting an overtwisted contact structure, and in Section~\ref{sec:OT} we characterise the Morse diagrams which are produced by negative stabilisation as those which have a \textit{left-veering handle}. 

\begin{theorem}\label{thm:OT}
A contact manifold is overtwisted if and only if it admits a Morse diagram with a left-veering handle.
\end{theorem} 

This result is similar in spirit to other criteria for overtwistedness, but we remark that a left-veering handle detects an overtwisted disc transverse to the pages. 

Finally, in Section~\ref{sec:monodromy} we consider open books with one-holed torus pages, a class of contact manifolds studied also in \cite{Baldwin}, \cite{Lisca}, \cite{TY}.  In this restricted setting, we are able to realise the aim described above more fully.  Specifically, we present a finite set of relations between Morse diagrams that we call \textit{Morse moves}, which may be seen in Figure~\ref{fig:MorseMoves}.

\begin{theorem}\label{prop:OBDMorseEquivalence}
Two Morse diagrams represent the same open book  if and only if they can be related by a sequence of the moves from Figure~\ref{fig:MorseMoves}.
\end{theorem}

As a consequence, we show how to translate between a Morse diagram and a factorisation of the monodromy as a produce of standard generators of the mapping class group of the page.  Although surface mapping class groups admit finite presentations, the relations are notoriously involved; extending this result to more general pages would provide an alternative tool for studying surface mapping classes.

\subsection{Acknowledgments}
 Results in this paper were included in the PhD thesis of the first author.

\section{Combinatorial Morse structures} 
 This section introduces \textit{combinatorial Morse structures}.  As the name suggests, these are combinatorial objects that encode the smooth \textit{Morse structures} introduced in \cite{GL}.  Here, we present a self-contained definition and summary of the essential properties, and we refer the reader to \cite{GL} for proofs.

\subsection{Open books and contact structures}

This section contains standard background on open book decompositions and contact structures.  See \cite{Etnyre} for more information.

\begin{definition}
An \emph{abstract open book} is a pair $(\Sigma,\phi)$, where $\Sigma$ is an oriented surface with boundary and $\phi\in \text{Diff}^{+}(\Sigma,\partial\Sigma)$. The map  $\phi$ is called the \emph{monodromy}.
\end{definition}

Generally, we will only care about the class of $\phi$ in $\text{Mod}(\Sigma)$, and we will use the term \emph{monodromy} for the mapping class, as well. An abstract open book $(\Sigma,\phi)$ determines a closed oriented 3-manifold $M(\Sigma,\phi)$ which is a quotient of the mapping torus
\[
M(\Sigma,\phi) = \frac{\Sigma\times[0,1]}{(x,1)\sim(\phi(x),0)}
\]
by the relation
\[
 (x,t)\sim(x, t') \text{ for } x \in \partial\Sigma,  t,t'\in [0,1].
\]

The image of $\partial \Sigma \times [0,1]$ is an oriented link called the \emph{binding} and the image of each surface $\Sigma\times \{t\}$ is a \emph{page}.

A \textit{contact structure} on a smooth 3-manifold is a nowhere integrable 2-plane field.  The compatibility between a contact structure and an open book decomposition requires the plane field to be close to the tangent plane field of the pages while still nowhere integrable.  To make this precise, recall that a vector field is \textit{contact} if its flow preserves a contact structure $\xi$.

\begin{definition}\label{def:supp}
The manifold $M(\Sigma, \phi)$ \emph{supports} the contact structure $\xi$ if the binding is positively transverse to $\xi$ and there exists a contact vector field positively transverse to both $\xi$ and to each page.
\end{definition}

Definition~\ref{def:supp} supports both computation and visualisation of contact structures, but the present study assumes this relationship and focuses on developing  combinatorial tools to exploit it more effectively.  Work of Thurston-Winkelnkemper and Giroux establishes that every open book supports a unique isotopy class of contact structures; the reader is referred to \cite{Etnyre} for a more thorough discussion  \cite{TW}, \cite{Gir}.  We write $M(\Sigma, \phi, \xi)$ when the manifold constructed from $(\Sigma, \phi)$ is viewed as a contact manifold for some  supported structure $\xi$. 

Contact structures are classified as either \textit{tight} or \textit{overtwisted}.  Motivating this dichotomy is beyond the scope of this paper, and we refer the reader to \cite{Geiges} for further discussion.   We will, however, examine how the properties of being tight or overtwisted are reflected in the supporting open books and their Morse diagrams.  This is discussed in Sections~\ref{sec:stab} and \ref{sec:OT}.

\subsection{Murasugi sums and open books}\label{sec:MuraTop}
 \textit{Murasugi sum} is a classical operation on surfaces with boundary.  In this section we recall the relationship between Murasugi sums, open books, and contact structures.  While the results are all established, we introduce notation that will be used again in Section~~\ref{sec:sums}.

\begin{definition}\label{def:murasum}
Let $\Sigma$ be a compact oriented surface with boundary and let $D$ be a $2n$-gon with edges sequentially labeled $p_{1}, q_{1}, \dots, p_{n}, q_{n}$. A \textit{Murasugi polygon} is an embedding of $D$ in  $\Sigma$ so that for all $k\in \{1, \dots, n\}$,
\begin{itemize}
\item $p_{k}\subset\partial\Sigma$; and
\item $\text{int}(q_{k})\subset \text{int}(\Sigma)$.
\end{itemize}
\end{definition}

Suppose $D_i$ is Murasugi polygon in $\Sigma_i$ for $i=1,2$.  The  \emph{Murasugi sum} $\Sigma_{1}*\Sigma_{2}$ is the surface  formed by identifying the polygons so that  $p^{1}_{i} = q^{2}_{i}$  and  $q_{i}^{1} = p_{i+1}^{2} $ for  all  $k$.  (Here and subsequently, indices are taken modulo $n$.)
The  sum $\Sigma_{1}*\Sigma_{2}$ depends on both the choice of $2n$-gons and the labels, but we suppress these in the notation for legibility.

\begin{figure} 
    \centering
    \includegraphics[width = 10cm]{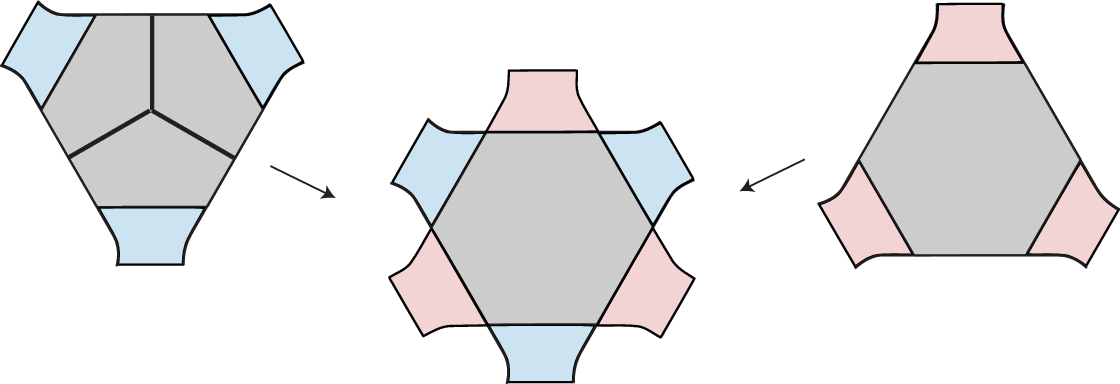}
    \caption{Identifying a pair of Murasugi polygons in surfaces $\Sigma_i$ (left, right) produces the Murasugi sum $\Sigma_1*\Sigma_2$ (center). The Murasugi polygon on the left is shown with its starlike graph, as described in Section~\ref{sec:sums}.} 
    \label{fig:murasugiGlue}
\end{figure}

Given two open books $(\Sigma_i, \phi_i)$, one may construct an open book whose pages are Murasugi sums:
 \[(\Sigma_{1},\phi_{1})*(\Sigma_{2},\phi_{2}) = (\Sigma_{1}*\Sigma_{2},\overline{\phi}_{1}\circ\overline{\phi}_{2}).\] Here $\overline{\phi}_{i} \in \text{Mod}(\Sigma_{1}*\Sigma_{2})$ is the extension of $\phi_i\in \text{Mod}(\Sigma_i)$  by the identity on $(\Sigma_{1}*\Sigma_{2})\setminus\Sigma_{i}$.

\begin{theorem}
 \cite{Gab} The manifold $M\big((\Sigma_{1},\phi_{1})*(\Sigma_{2},\phi_{2})\big)$ is diffeomorphic to the connect sum $M(\Sigma_{1},\phi_{1})\#M(\Sigma_{2},\phi_{2})$.
\end{theorem} 

Thus the choice of $2n$-gons affects the open book, but not the resulting manifold.    

As one might hope, this operation respects  the supported contact structures.

\begin{theorem}\label{thm:TorisuConnectSum} \cite{Tor} The contact structure supported by the connect sum of open books is isotopic to the contact structure formed by taking the connect sum of the contact structures on each of the factors: 
\[ \big(M(\Sigma_{1} * \Sigma_2,\overline{\phi}^{1}\circ \overline{\phi}^{2}), \xi_{\overline{\phi}^{1}\circ \overline{\phi}^{2}}\big)= \big(M(\Sigma_{1},\phi_{1}), \xi_{\phi_{1}}\big)\# \big(M(\Sigma_{2},\phi_{2}), \xi_{\phi_{2}}\big).\]

\end{theorem}

In contact geometry, as in topology more generally, the operation of connect sum has a unit.  Let $\xi_{\text{std}}$ denote the unique tight contact structure on $S^3$.  Then for any contact manifold $(M, \xi)$, we have
\[ (M, \xi) \# (S^3, \xi_{\text{std}})=(M, \xi).\]

The simplest nontrivial instance of this operation in the context of open books is the open book connect sum with the annular open book for the tight $3$-sphere, $(A^2, \tau)$.  Here, $\tau$ is a positive Dehn twist around the core of the annulus.  It follows that the open books 

\[ \big(M(\Sigma,\phi), \xi_{{\phi}}\big)   \text{          and          }   \big(M(\Sigma * A^2,\overline{\phi}\circ \overline{\tau}), \xi_{\overline{\phi}\circ \overline{\tau}}\big) \]
support the same contact structure.   The modification of an open book by connect sum with this unit is called \textit{positive stabilisation}.  The  \textit{Giroux Correspondence}, which is one of the most useful structure theorems in contact topology, states that all open books supporting a fixed contact structure are related by a sequence of positive stabilisations and destabilisations \cite{Gir}. 

\subsection{Combinatorial Morse structures}

Let $\Sigma$ be a surface with non-empty boundary and fix a basepoint on each component of $\partial \Sigma$.
\begin{definition}\label{def:hs} A \emph{handle structure} on $\Sigma$ is a collection of oriented simple closed curves, called \emph{cores}, and properly embedded oriented arcs, called \emph{co-cores}, that satisfy the following properties:
\begin{enumerate}
    \item\label{1} the cores form a wedge of circles; 
    \item $\Sigma$ deformation retracts onto the union of the cores; and
    \item\label{3} each core $A$ intersects exactly one co-core $A^*$, oriented so that the signed intersection number $i(A^*,A)$ is positive.  The endpoints of the co-cores  lie in the complement of the basepoint(s) on $\partial \Sigma$.
\end{enumerate}

 Two handle structures are \textit{equivalent} if they are isotopic as properly embedded graphs on $\Sigma$ in the complement of the basepoint(s). 

\end{definition}

Label the initial endpoint of the oriented co-core $A^*$ by $A_-$ and the terminal endpoint by $A_+$.  Given a handle structure, one may record the order in which these endpoints appear around $\partial \Sigma$, beginning from the basepoint and following the positve boundary orientation.  If $|\partial\Sigma |>1$, the order of the boundary components should be specified.  We call this set of ordered lists a \emph{boundary configuration}. Note that the homeomorphism type of the surface can be recovered from the boundary configuration alone: surger the set of circles indexed by basepoints by attaching a 1-dimensional 1-handles along each $S^0$ with matching labels.  This creates a new set of circles which may be capped off by discs to produce $\Sigma$.

\begin{example} Figure~\ref{fig:handleex} shows three different handle structures on the thrice-punctured sphere.  Labeling the blue curves by $A$ and the red by $B$, the boundary configuration on the left is
\[ [A_+, B_-], [A_-], [B_+],\]
while the boundary configurations on the centre and right are each
\[ [B_-], [A_-],[B_+, A_+].\]
\begin{figure}[h] 
    \centering
    \includegraphics[width = 12cm]{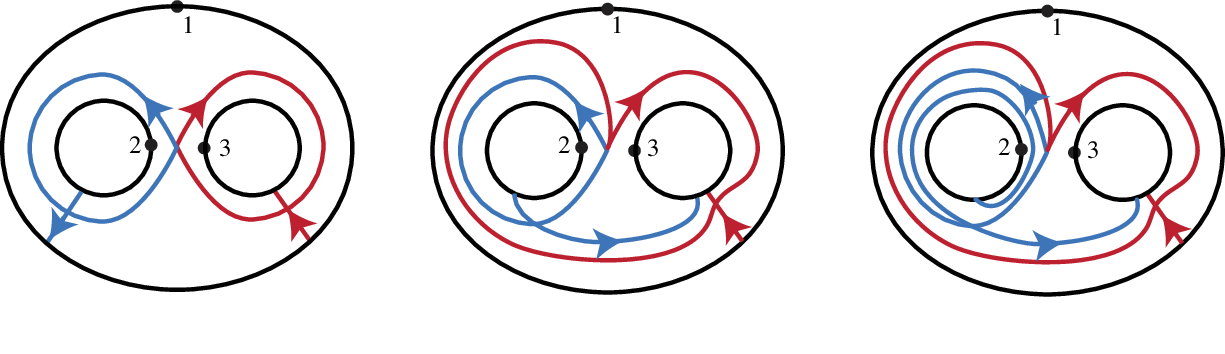}
    \caption{ Distinct handle structures on the thrice-punctured sphere.} 
    \label{fig:handleex}
\end{figure}
\end{example}

A fixed surface admits infinitely many inequivalent handle structures, but   \textit{arc slides} of the co-cores suffice to transform one handle structure to another.  In an arc slide, the \textit{sliding end} of a co-core $A^*$ approaches the \textit{initial end} of another co-core $B^*$; traverses $B^*$; and slides off the \textit{final end} of $B^*$.  This move preserves all the cores except $B$; up to isotopy, there is a unique way to complete the handle structure with a new $B$ satisfying the  duality condition.  In Figure~\ref{fig:handleex}, performing an arc slide of $A_+$ over $B_-$ on the right-hand picture yields the centre picture.

\begin{figure}[h]
    \centering
    \includegraphics[width = \textwidth]{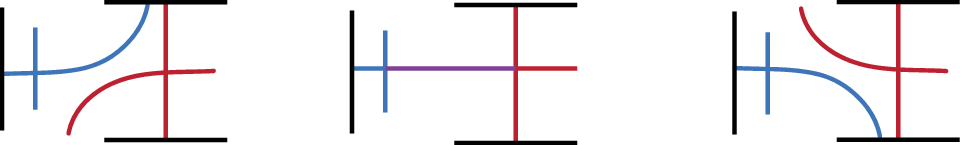}
    \caption{The left and right pictures are local images of handle structures $\H_i$ and $\H_{i+1}$ that are related by an arc slide of the blue co-core over the red one; the central figure shows the singular handle structure $S_{i+1}$ connecting them. }
    \label{fig:LocalHandleslide}
\end{figure}

As described, an arc slide is continuous process, but the equivalence classes of the initial and final handle structures are determined by discrete data.  Specifically,  we associate a  \emph{singular handle structure} to each arc slide as follows.  With the notation as above, remove the portion of $B$ between the wedge point and the sliding side of $B^*$ and fix the sliding endpoint of $A^*$ at $B^*\cap B$.  This configuration does not satisfy Definition~\ref{def:hs}, but instead determines a pair of handle structures related by an arc slide.

The requirement that co-core endpoints lie in the complement of the basepoint introduces one further move, which is an isotopy of a co-core endpoint across the basepoint.  We will also refer to this as an arc slide, and in the associated singular handle structure the sliding endpoint of the co-core lies on the basepoint. An arc slide across a basepoint preserves the cores of the handle structure.

 The centre and right pictures in Figure~\ref{fig:handleex} are related by an arc slide across the basepoint labeled 2.

\begin{proposition}\label{prop:morseob}
Any two handle structures on $\Sigma$ are related by a finite sequence of arc slides.  If $\H, \H'$ are two handle structures on $\Sigma$ inducing the same boundary configuration, then there is a unique  $\phi\in \text{Mod}(\Sigma)$  such that $\phi(\H')$ is isotopic to $\H$.
\end{proposition}

\begin{proof}  Given a handle structure, one may construct a Morse-Smale function whose critical flowlines  recover the (unoriented) equivalence class of the handle structure.  Here, the wedge point is an index $0$ critical point; each intersection $A\cap A^*$ is an index $1$ critical point; and $\partial \Sigma$ is a maximal level set.  Cerf theory implies that the Morse-Smale functions associated to a distinct pair of handle structures are related by a path of functions whose associated handle structures undergo a sequence of arc slides that transform the initial one to the final one.  See, for example, \cite{GK}.

The proof that $\phi$ is unique is a consequence of the Alexander Method, which ensures that two mapping classes that act identically on a sufficiently complicated set of curves are equal.  See, for example, Proposition 2.8 in \cite{FM}.
\end{proof}

\begin{remark} This proof is taken from  Lemma 4.5 and Proposition 2.8 in \cite{GL}, where each handle structure is induced by a Morse-Smale function on $\Sigma$.
\end{remark}

\begin{definition}
    A \emph{combinatorial Morse structure} on $\Sigma$ is a finite sequence  \[\H_0, S_1, \H_1, \dots, 
 S_n, \H_n\] such that (1) each $\H_i$ is a handle structure on $\Sigma$; (2) each $S_i$ is a singular handle structure transforming $\H_{i-1}$ into $\H_i$; and (3) $\H_0$ and $\H_n$ induce the same boundary configuration.
\end{definition}

Given an initial handle structure, each arc slide determines the succeeding handle structure.  A combinatorial Morse structure therefore encodes the homeomorphic final handle structure, and hence, the monodromy of the open book.   If the initial handle structure isn't specified, one may choose an arbitrary handle structure compatible with the boundary configuration of $\H_0$ and compare this to its image under the arc slides associated to  the combinatorial Morse structure; this determines the monodromy up to conjugation, and hence, the open book.

\subsection{Morse diagrams}

A Morse diagram is a graphical representation of a combinatorial Morse structure, and this will be our preferred method for encoding this data.

Given a combinatorial Morse structure, associate an ordered list of rectangular diagrams called an \textit{elementary slice} to each subsequence $\mathcal{H}_i {S}_{i+1}\mathcal{H}_{i+1}$  as follows. The rectangles are indexed by boundary components of $\Sigma$, and on each rectangle
\begin{enumerate}
\item  the the vertical edges are labeled by the corresponding basepoint; 
\item  points along the bottom edges (oriented left to right) are labeled as the points in the $\mathcal{H}_i$ boundary configuration; and
\item points along the top edges are labeled as the points in the $\mathcal{H}_{i+1}$ boundary configuration.
\end{enumerate}

If two points on the horizontal edges are labeled as endpoints of a non-sliding end of a co-core, connect them by a curve with no horizontal tangents.  (We will refer to these as vertical curves.) Finally, recall that in the singular handle structure, the moving co-core connects to the centre of another co-core.  Both the initial and the terminal end of this co-core are now represented by  vertical curves.  Connect the remaining basepoint on the bottom  to the midpoint of the vertical edge labeled by the initial end and connect the remaining basepoint on the top  to the midpoint of the vertical edge labeled by the final end.  These connecting curves, too, must have no horizontal tangents, but any isotopy preserving all the above conditions is allowed.

While precise, this description is disproportionately unwieldy given the simplicity of the diagrams produced.  See Figure~\ref{fig:slice} for some examples. As shown there, we will often indicate the pairing on trace curves via colour, rather than alphabetic labels. 

\begin{figure}[h]
    \centering
    \includegraphics[width = \textwidth]{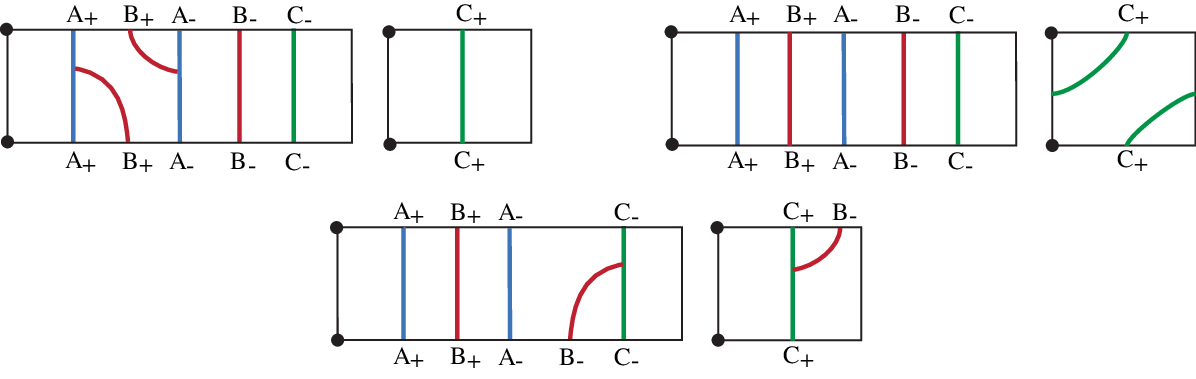}
    \caption{Three elementary slices on a two-holed torus.  Observe that the elementary slices in the top row can be stacked in any order because the initial and final handle structures agree, but the elementary slice in the bottom row can only be stacked on top of one of the slices from the top row.}
    \label{fig:slice}
\end{figure}

Stacking these elementary slices  produces a planar diagram. We perform horizontal isotopy of the endpoints as required to identify curves from different slices.

\begin{definition} The \textit{Morse diagram} for the combinatorial Morse structure \[\H_0, S_1, \H_1, \dots, S_n, \H_n\] is the vertical concatenation of the elementary slices associated to $(\H_0, S_1, \H_1)$,   $(\H_1, S_2, \H_2),  \dots (\H_{n-1}, S_n, \H_n)$.
\end{definition}

After concatenating, we allow  ``braid-like'' isotopies preserving the property that the labeled curves meet constant-$t$ horizontal curves transversely.  Morse diagrams will always be understood up to this equivalence, and for convenience, we also rescale the vertical coordinate to $[0,1]$ to create a \emph{planar Morse diagram}.   Identifying the left and right edges (but retaining the curve labeled by the basepoint) produces an \emph{annulur Morse diagram}, and further gluing the top and bottom circles yields a \emph{toroidal Morse diagram}.  In each case, we say that the Morse diagram is decorated with paired \emph{trace curves} labeled by the corresponding endpoints of the co-cores.  

\begin{example}\label{ex:twist} Figure~\ref{fig:case} shows a sequence of handle structures on a one-holed torus together with the corresponding Morse diagram.
\begin{figure}[h]
    \centering
    \includegraphics[width = \textwidth]{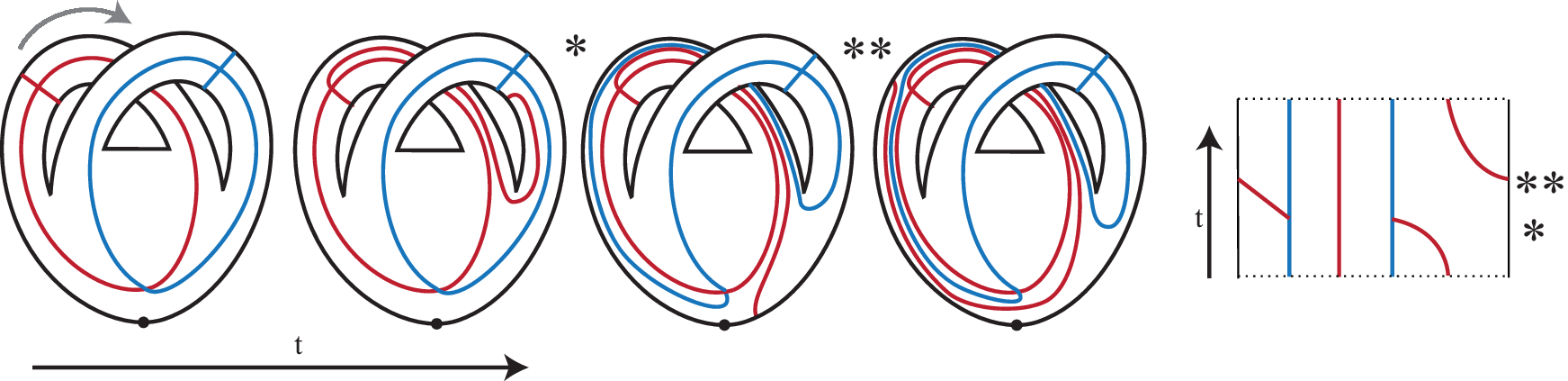}
    \caption{ A sequence of handle structures and the corresponding planar Morse diagram. }
    \label{fig:case}
\end{figure}

The grey arrow indicates how the endpoint of the red co-core moves between the first and second image, traveling against the boundary orientation of the surface.  The single asterisk marks where this moving endpoint slides over the blue co-core, while the double asterisk marks where this endpoint crosses the basepoint before returning to its starting point.  

Observe that the final diagram differs from the initial diagram by a positive Dehn twist around the red core.
\end{example}

    \begin{figure}[h!]
        \centering
        \includegraphics[width = 12cm]{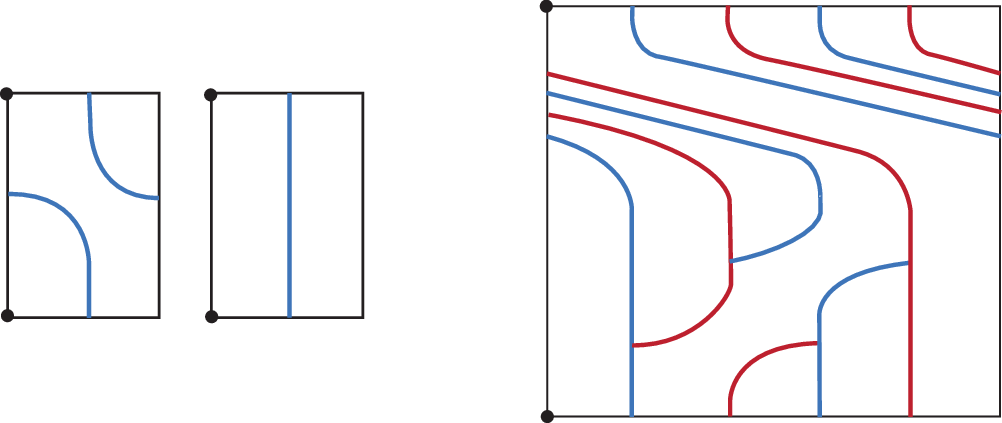}
        \caption{Left: A Morse diagram for an open book with annular pages.  Right: A Morse diagram for an open book with one-holed torus pages.}
        \label{fig:PHSexample}
    \end{figure}

We recall one further term from \cite{GL} that will be useful when discussing Morse diagrams.  Suppose that some height $t_0$ on a Morse diagram corresponds to a singular handle structure where the co-core $A^*_\pm$ arc slides over a co-core $B^*$.   For some $t<t_*$, the  trace curve corresponding to $A^*_\pm$ approaches one of the $B^*$ trace curves.  It collides with this $B^*$ trace curve and emerges from the other at  $t=t_*$ and then it moves away from the latter trace curve as $t$ increases above $t_*$.  In this case, we say that $A^*_\pm$ \textit{teleports} across the $B^*$ pair at $t_*$.

\begin{example}

   The left-hand side of Figure~\ref{fig:PHSexample} shows a planar Morse diagram for an open book with annular pages.  The starting and ending boundary patterns are identical, but the interpolating singular handle structure occurs as one end of the co-core traverses a component of the boundary in the negative direction and crosses the basepoint.  The monodromy of the open book is the mapping class that takes $\H_1$ to $\H_0$, so this is a negative Dehn twist around the core.

    \end{example}

    \begin{example}
   The right-hand side of Figure \ref{fig:PHSexample} shows a more complicated Morse diagram with six arc slides, four of which are across the basepoint on $\partial \Sigma$.   Example~\ref{ex:phs}  shows that the manifold associated to this Morse diagram is the Poincar\'{e} homology sphere with reversed orientation.  

\end{example}

\section{Connect sums of open books}\label{sec:sums}

Having introduced Morse diagrams, the remainder of this paper examines how they can be used as tools to study contact manifolds.  In this section, we consider connect sums of open books with Morse structures.

\subsection{Murasugi sums for Morse structures}\label{sec:Murasugi}
In this section we  extend the connect sum operation on $3$-manifolds to an operation on combinatorial Morse structures on open books.  It may be helpful to refer to Figure~\ref{fig:splice}.  

Assume a handle structure is fixed on $\Sigma$ and choose a set of $n$ points embedded on  $\partial \Sigma$ in the complement of the basepoint(s).  Since the co-cores cut $\Sigma$ into discs, there is a unique (up to isotopy) way to connect each point to the wedge point via arcs disjoint from the rest of the handle structure.  The resulting graph is a star and induces a cyclic order on the points.   We call a set of cyclically ordered points $\{p_j\}_{j=1}^n$ on $\partial \Sigma$ \textit{starlike} if it arises thus.

A regular neighbourhood of this star can be  cornered to produce a Murasugi polygon for $\Sigma$, as shown in Figure~\ref{fig:murasugiGlue}.  Observe that the label $p_i$ applies interchangeably to the edge of the polygon or to the point on the boundary.  

In fact, we can realise any Murasugi polygon via this construction.

\begin{proposition}\label{prop:MakeNice} 
If $D$ is a Murasugi polygon on $\Sigma_i$, there exists a handle structure such that the points $p_j^i$ are starlike.   
\end{proposition}

We defer the proof of this proposition to the end of the section.

The property of a cyclically ordered set being starlike depends only on the boundary pattern:  two handle structures inducing the same boundary pattern are related by a homeomorphism of $\Sigma$ that would take a star compatible with the first handle structure to star compatible with the second handle structure.  However, the specific polygon  depends on the handle structure, so it is useful to treat the handle structure as known.

We next describe how to build a Morse diagram for $M(\Sigma_1,\phi_1)\#M(\Sigma_2, \phi_2)$ directly from Morse diagrams for the summands. The procedure is combinatorial: we cut the original annular Morse diagrams into rectangular pieces,  arrange them in a checkerboard configuration  separated by blank rectangles, and  carefully extend the original trace curves.    It may be helpful to refer to Figure~\ref{fig:splice}.

\textbf{Step 1: Cutting}
Suppose $\{p_j\}$ is a starlike set with respect to $\mathcal{H}_0$ and let $X$ be a planar Morse diagram.  Note that the indexed order of the points need not agree with the horizontal coordinate on $X$.  Label the interval to the right of $p_j$ by $r_j$.   The right-hand endpoint of $r_j$ is some $p_k$; define $\sigma\in S_n$ by $\sigma(j)=k$.  Now cut $X$ along the vertical line associated to each $p_j$ to produce  $4n$ rectangles whose horizontal edges are labeled by the intervals $r_i$.  

\textbf{Step 2: Arranging}
Perform Step 1 on the two Morse diagrams, using superscripts $i=1,2$ to distinguish them.  The rectangles from $X^1$ will all be placed on the bottow row of a checkerboard and the rectangles from $X^2$ will be placed on the top.  Align the right edge of the rectangle labeled by $r_j^1$ with the left edge of the rectangle labeled by $r^2_{\sigma^1(j)}$ and align  the right edge of the rectangle labeled by $r_j^2$ with the left edge of the rectangle labeled by $r^1_{\sigma^2(j+1)}$.    In general, this  produces multiple annular checkerboards.

\textbf{Step 3: Extending} Extend the trace curves across the empty squares of the checkerboards as follows:
\begin{enumerate}
\item for each endpoint of a trace curve on a horizontal boundary of a square, extend it vertically to the matching endpoint on the parallel edge of the square.
\item for each endpoint of a trace curve on a vertical boundary of an empty square, extend it horizontally to meet the matching endpoint on the parallel edge of the square.
\item\label{step3} for each endpoint of a trace curve on a vertical boundary of a non-empty square, extend it horizontally until it meets a vertical trace curve and then teleport; repeat until the curve meets the matching endpoint on the parallel edge of the original square.
\end{enumerate}

Isotope the diagram so that each  trace curve is once again monotonic with respect to the $t$-coordinate and select one basepoint for each component of the diagram. The resulting annular Morse diagram is the \emph{splice} of $X^1$ and $X^2$ with respect to the initial data, and we denote it by $X^1*X^2$.

\begin{example}\label{ex:MurasugiToriExample}
We take as our starting point the Morse diagrams from Figure~\ref{fig:PHSexample}, although we simplify the right-hand diagram by removing the boundary parallel twist at the top.  

\begin{center}
\begin{figure}[h]
  \includegraphics[width=10cm]{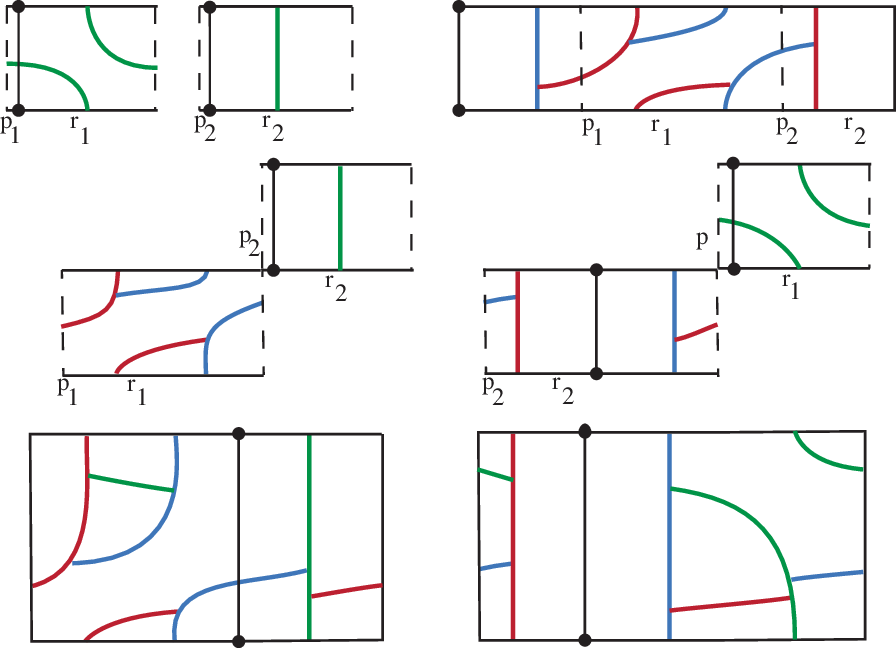}
  \caption{ Splicing the Morse diagrams in the top row produces the two-component annular diagram in the third row.}
  \label{fig:splice}
\end{figure}
\end{center}

The first row in Figure~\ref{fig:splice} shows dotted vertical lines for each $p_i$ and labels the segments by $r_i$.  The permutation associated to the left-hand diagram is the identity, while the permutation associated to the right-hand diagram is $(12)$.  After arranging the rectangles in a checkerboard in the centre row, we extend the trace curves according to the rules in Step 3 above.
\end{example}

This formal operation on Morse diagrams produces a Morse diagram for the connect sum of the original manifolds.

\begin{theorem} When  $X^1$ and $X^2$ are Morse diagrams for contact manifolds $(M^i, \xi^i)$, the splice $X^1*X^2$ is a Morse diagram for $(M^1, \xi^1)\# (M^2, \xi^2)$.  
\end{theorem}

As a first step, we verify that the page associated to the spliced diagram is a Murasugi sum of the original pages.

\begin{lemma} Suppose that $X^1*X^2$ is a splice of Morse diagrams $X^1$ and $X^2$.  Then the page in the open book represented by $X^1*X^2$ is a Murasugi sum of the pages associated to $X^1$ and $X^2$ along the polygons associated to the starlike points. 
\end{lemma}

\begin{proof}  As noted above, a neighbourhood of each set of  starlike $\{p_i\}$  is a Murasugi polygon.  
Since the two stars intersected the original handle structures only at the wedge point, the superposition of the handle structures produces a handle structure for the new surface.  To see that the associated boundary pattern agrees with the spliced diagram, observe that the boundary of the new surface is built from alternating intervals of the two original boundaries, with endpoints on the $\{p_i\}$.  The permutations $\sigma^1$ and $\sigma^2$ are chosen  to order these intervals correctly. 
\end{proof}

\begin{figure}[h]
    \centering
    \includegraphics[width = 10cm]{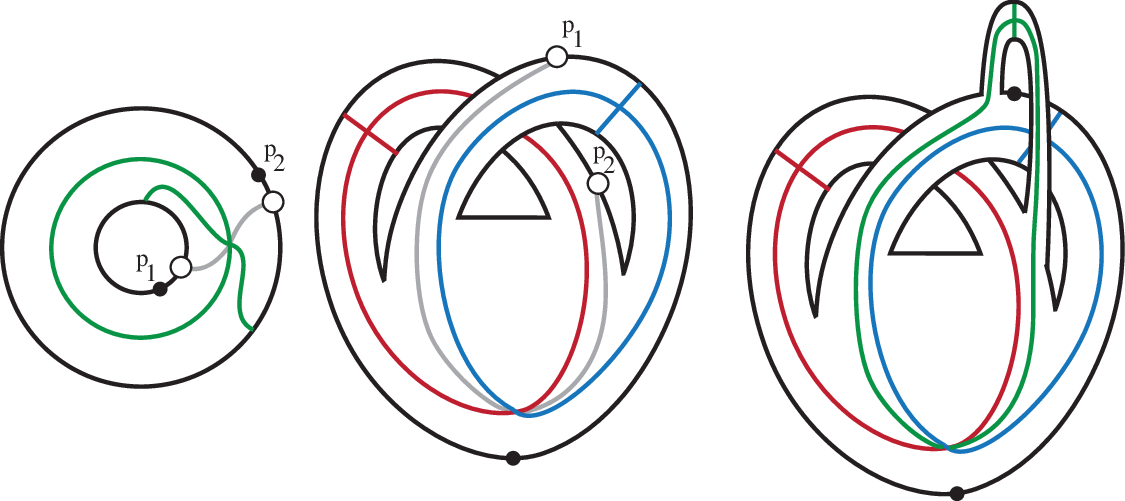}
    \caption{ The grey lines on the left and centre pictures indicate Murasugi polygons.  The right picture is the sum of the first two along these polygons. The bottom row of Figure~\ref{fig:splice} shows a Morse diagram for an open book with this page.}
    \label{fig:sumex2}
\end{figure}

\begin{proof}[Proof of Theorem~\ref{thm:splice}]  Having  established the topology of the page associated to $X^1*X^2$, we turn attention to the monodromy.
Observe that $t=\frac{1}{2}$ in the spliced diagram corresponds to $t=1$ from the lower factor and $t=0$ from the upper factor.  Thus at $t=\frac{1}{2}$, each of the identified Murusugi polygons is a neighbourhood of a star on $\Sigma^1*\Sigma^2$.  

We will verify that the arc slides occurring in $\frac{1}{2}\leq t\leq 1$ determine $\overline{\phi}^2$.  Each trace curve associated to a co-core of $\Sigma^1$ is vertical, so the monodromy acts trivially on $(\Sigma^1*\Sigma^2)\setminus \Sigma^2$.  
In order for the given arc slides on $X^1*X^2$ to determine $\overline{\phi}^2$, we require that the final handle structure at $t=1$, restricted to $\Sigma^2$, matches the final handle structure produced by the original $X^2$. If an arc slide from $X^2$ occurs along a segment of the boundary that persists in $X^1*X^2$, then the corresponding arc slide appears on $X^*X^2$ and changes the handle structure identically.  However, an arc slide from $X^2$ may also occur along a segment of the boundary containing a point $p_j^2$.  In this case, the sliding endpoint of the co-core will pass into a segment of $\partial(\Sigma^1*\Sigma^2)$ coming from $\Sigma^1$. 
Since we may recover $\Sigma^2$ from $\Sigma_{1} * \Sigma_2$ by cutting along each co-core of $\phi^1$, the moving endpoint of the co-core will reach the other copy of $p_j^2$ along $\partial(\Sigma^1*\Sigma^2)$ if it slides across each co-core of $\Sigma^1$ that it encounters.  This is exactly the behaviour dictated by Step~\ref{step3} in the construction of the splice.  It follows that at $t=1$, the handle structure on $\Sigma^2\subset(\Sigma^1*\Sigma^2)$ matches the handle structure at the $t=1$ slice of $\Sigma^2$ in the original open book, as desired.  

The argument applies essentially verbatim to show that $\overline{\phi}^1$ is determined by the lower half of $X^1*X^2$; the only difference is that in this case, the polygon is a neighbourhood of rays occurring at the end of the sequence of arc slides associated to $\phi^1$, rather than at the beginning. \end{proof}

Finally, we prove Proposition~\ref{prop:MakeNice}.

\begin{proof}[Proof of \ref{prop:MakeNice}]
Let  $D$ denote the star-shaped graph associated to the Murasugi polygon and let $v$ denote the $n$-valent vertex.   We will build a handle structure $\H$ for $\Sigma$ such that   $D\cap\H=v$. 
    
    Enumerate the closures of the components of $\Sigma\setminus\delta$  as $\Sigma_{1},\dots,\Sigma_{k}$, allowing $k=1$ if $\Sigma\setminus\delta$ is  connected. The boundary of each $\Sigma_{i}$ has arcs that glue to $\delta$; let $q_j$ be the points on these arcs that glue to $q\in\Sigma$.
    
   For each component $\Sigma_i$, choose a handle structure.  Label the wedge point for the handle structure on $\Sigma_i$  by $v_i$.  Each $\Sigma_i$ has one or more copies of $v$ on its boundary; isotope the handle structure on each component so that $v_i$ lies on a copy of $v\in \partial \Sigma_i$.  For each unused copy of $v\in \partial \Sigma_i$, choose a properly embedded arc whose interior is disjoint from the handle structure and that connects $v$ to $v_i$.  We claim that the image of these curves and the component-wise handle structures include a complete set of cores for $\Sigma$.
      
 By construction, the points $v_i$ are all identified with $v$, forming a single wedge point for the cores.  To see that we have enough cores for a handle structure, let $C$ be a simple closed curve in $\Sigma$.  Each time $C$ crosses an edge of $D$, isotope $C$ so that it crosses at $v$.  This decomposes $C$ into a collection of arcs, each of which is either a closed curve or a properly embedded curve in some $\Sigma_i$.  Each closed component deformation retracts onto the cores of the corresponding handle structure, while each curve with distinct endpoints in $\Sigma_i$ is isotopic to a concatenation of the added curves.  Finally, co-cores can be added to produce the desired handle structure.  
 \end{proof}

\subsection{Stabilisation}\label{sec:stab}
As noted in Section~\ref{sec:MuraTop},  the most important example of connect sums in the context of open books is stabilisation.  Recall that the open book $(\Sigma, \phi) \# (A^2, \tau)$ is a \emph{positive} (respectively, \emph{negative) stabilisation} of $(\Sigma,\phi)$ whenever $\tau$ is a positive (negative) Dehn twist around the core of the annulus.  Positive stabilisation preserves the supported contact structure, while negative stabilisation does not.  Here, we examine stabilisation in the context of the construction above.

A product neighbourhood of any properly embedded arc on $\Sigma$ is a Murasugi 4-gon.  Any pair of vertical lines on a planar Morse diagram associated to a known handle structure defines a stabilisation; we differentiate between stabilisations along different arcs with shared endpoints by choosing different handle structures.  The splicing operation is particularly simple in the $n=1$ case, as seen next.

\begin{example}\label{ex:negstab}
Here, we examine how negative stabilisation changes the Morse diagram.  The open book with binding a negative Hopf link is has the Morse diagram shown on the left in Figure~\ref{fig:PHSexample}.  One trace curve is vertical, while the other veers left as $t$ increases and crosses the basepoint once.  When the first component is included in a splice, the trace curve extends vertically through the new Morse diagram.  When the second component is included in a splice, the curve veering to the left will continue veering to the left, teleporting across any trace curve it encounters until it returns to its original horizontal coordinate.  This behaviour -- which can be seen in the green curves in Figure~\ref{fig:splice} -- will play a central role in the next section.
\end{example}

\subsection{Morse diagram for overtwisted contact manifolds}\label{sec:OT}

In this section we use Murasugi sums to characterise Morse diagrams associated to overtwisted contact manifolds.  

\begin{definition}\label{def:left_veering} A \textit{left-veering handle} on a Morse diagram is a trace curve pair such that 
\begin{itemize}
\item one of the paired curves is vertical; and
\item the $x$-coordinate of the other paired curve is monotonically non-increasing, except for possibly at  points where it teleports. 
\end{itemize}
\end{definition}

\begin{theorem}\label{thm:OT}
A contact manifold is overtwisted if and only if it admits a Morse diagram with a left-veering handle.
\end{theorem} 

The proof of Theorem~\ref{thm:OT} relies on two ingredients from the work of Honda, Kazez, and Matic \cite{HKM}.  First,  they show that a contact manifold is overtwisted if and only if it decomposes as a connect sum of some other contact manifold with an overtwisted $S^3$.  In the same paper, they introduce the notion of a \textit{right-veering diffeomorphism}, whose key properties we recall here.   Using the formulation from  Definition 3.1 in  \cite{IK},  a properly embedded arc $\gamma\subset \Sigma$ is \textit{right-veering} if one of the two following conditions holds: $\phi(\gamma)$ is isotopic to $\gamma$ or if, after isotoping $\gamma$ and $\phi(\gamma)$ to intersect minimally, $\phi(\gamma)$ lies to the right of $\gamma$ at each shared endpoint.

\begin{proposition}\label{prop:hkm}[Theorem 1.1, \cite{HKM}] If there exists an arc $\gamma$ which is not right-veering for $\phi$, then the contact structure supported on $M(\Sigma, \phi)$ is overtwisted.
\end{proposition}

\begin{proof}[Proof of Theorem~\ref{thm:OT}]
We first show that every overtwisted contact manifold $M$ admits a Morse diagram with a left-veering handle.  As noted in Example~\ref{ex:negstab}, splicing a Morse diagram with the simplest Morse diagram for $(A^2, \tau^{-1})$ yields a Morse diagram with a left-veering handle.  According to the first result from \cite{HKM} noted above, $(M, \xi_{OT})$ decomposes as $(M, \xi_{OT})=(M', \xi') \# (S^3, \xi_{-1})$.  Choose any open book and Morse structure for the first factor and choose the simplest Morse diagram for the open book $(A^2, \tau^{-1})$ for the overtwisted sphere.  Splicing the two diagrams produces a Morse diagram for $M'$ that has a left-veering handle, as desired.

To show the other implication, suppose that a Morse diagram for $M(\Sigma, \phi, \xi)$  has a left-veering handle labeled $A^*_\pm$.  Then we claim that the co-core $A^*$ in $\H_0$ is non-right-veering in the sense of \cite{HKM}.   If there are no handle slides, the claim is obvious, so suppose $A^*$ has a moving trace curve $A^*_+$ and a stationary one $A^*_-$.  Observe that cutting $\Sigma$ along every co-core except $A^*$ produces an annulus spanned by $A^*$.  Each other co-core is recorded as a pair of marked intervals on the boundary, and the original surface can be reconstructed by gluing these paired intervals.

Suppose that no trace curves besides $A^*$ teleport on the Morse diagram.  We examine the movie of handle structures on $\Sigma$ dictated by the Morse diagram.  The condition that $A^*$ is a left-veering handle implies that the endpoint $A^*_0$ moves against the boundary orientation of this annulus as $t$ increases.  It follows that in the final handle structure $\H_n$, this co-core is the image of the original $A^*$ under some number of positive Dehn twists around $A$.  This is equivalent to the statement that the image of $A^*$ under $\phi$ lies to the left of $A^*$.  See Figure~\ref{fig:caselvh} for an illustrative example.

To complete the argument, we must show that the same reasoning holds when other trace curves teleport on the Morse diagram.  If no curves teleport across $A^*$, then the argument above applies with no variation.  For the general case, suppose that $B^*$ arcslides across $A^*$ at $t_*$.  In order to compare curves and their images as the movie progresses, we project $\Sigma\times [0,1]$ to a reference copy of $\Sigma$ and temporarily introduce the notation $A^*_t$ for the co-core from $\Sigma\times\{t\}$,  isotoped relative to its endpoints to have minimal intersection with $A^*_0$.

We will prove the following three claims:
\begin{enumerate}
\item\label{claim1} For all $t\neq t_*$, $A^*_t$ lies to the right of $A^*_0$ at their shared, fixed endpoint $A^*_-$.
\item\label{claim2} If (\ref{claim1}) holds for a $t$-interval in which a single co-core slides across $A^*$, then it holds for any number of arc slides across $A^*$. 
\item Together, Claims~\ref{claim1} and \ref{claim2} imply that whenever both endpoints of $A^*_t$ coincide with the endpoints of $A^*_0$, the arc $A^*_t$ is to the right of $A^*_0$ in the sense of Definition 3.2 in \cite{IK}. 
\end{enumerate}

The final claim shows that $A^*$ is a left-veering arc, as desired.

 For the first claim, consider the pair of pants or one-holed torus formed by cutting along every co-core except $B^*$ and $A^*$. For $t<t_*$, Claim~\ref{claim1} follows from the annular argument above.  For $t>t_*$, the claim follows from a  checking cases on these simple surfaces.
 
Figure~\ref{fig:caselvh} shows a single case on the pair of pants.   An endpoint $A^*_t$ initially moves against the boundary orientation, so Claim (\ref{claim1}) holds in the second frame.  In the next frame, $B^*$ has slid over $A^*$.  In the remaining frames the moving endpoint of $A^*$ travels against the boundary orientation and teleports until returning to its original position. We then check that Claim~\ref{claim2} holds.   The other cases on the pair of pants and on the one-holed torus are analogous. 
 
 \begin{figure}[h]
    \centering
    \includegraphics[width = \textwidth]{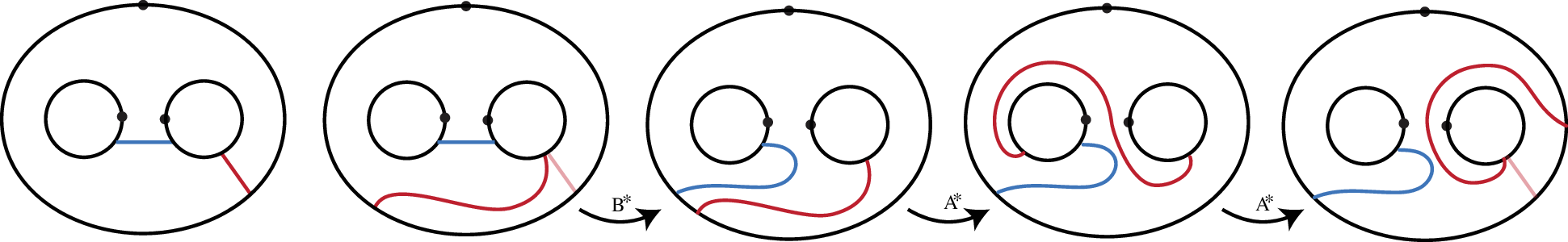}
    \caption{ The time parameter increases from left to right, and the label on an arrow indicates the co-core sliding between the two frames.}
    \label{fig:caselvh}
\end{figure}

 Claim~\ref{claim2} follows from the proof in \cite{IK} that  $\prec_{\text{right}}$ is a total order on arcs with a shared endpoint. 
 
 Definition 3.9  considers the image of an arc under a mapping class $\phi$, in which case both endpoints of $\gamma$ and $\phi(\gamma)$ coincide.  In our setting, many handleslides may occur before both endpoints of $A^*$ coincide. However, the definition of a  left-veering handle implies that  the moving endpoint $A^*_+\in A^*_t$ ultimately approaches $A^*_+\in A^*_0$ from the right, as well.  
\end{proof}

To conclude this section, we compare left-veering handles to Goodman's sobering arcs, which also certify that an open book supports an overtwisted manifold $(M, \xi)$.  A sobering arc guides the construction  of a surface $S\subset M$ that violates the Thurston-Bennequin inequality; in this case, the  Legendrian boundary $\partial S$ lies  near a page of the open book.  
 
 In contrast, a left-veering handle  guides the construction of an overtwisted disc transverse to the pages.  The construction relies on a result from \cite{GL} that associates a Legendrian unknot to the moving trace curve of the left-veering handle.  This unknot bounds a disc $D^2$ that intersects each page once in a Legendrian arc, and all these Legendrian arcs meet at the intersection with the binding.  This characteristic foliation identifies $D^2$ as an overtwisted disc.  The existence of an overtwisted disc punctured once by the binding was shown in Theorem 5.2.3 of \cite{BO}.  Their argument  required choosing an open book which is a negative stabilisation, whereas our construction makes no such assumption.

 When an open book has one-holed torus pages,  left-veering monodromy implies that the open book admits a Morse diagram with a left-veering handle.  
  However, we do not know if this holds in general: 

\begin{question}
    For a fixed $(\Sigma, \phi)$ with $\phi$ left-veering, does the open book necessarily admit a Morse diagram with a left-veering handle? 
   \end{question}

\section{Once-punctured torus pages}\label{sec:monodromy}
Techniques in the previous section raise the possibility of classifying all the combinatorial Morse structures for a fixed contact manifold, with positive stabilisation providing an example of an operation that preserves the contact manifold while changing the open book structure.  As a first step in such a classification, we might want to identify all the Morse structures compatible with a fixed open book. 

There are both obvious and less obvious alterations to a Morse diagram that preserve the page and the monodromy; an example of the first is inserting an arc slide and its inverse into any combinatorial Morse structure;  more complicated sequences of slides are suggested by Cerf theory.  Although we cannot yet enumerate a set of  moves that relate all combinatorial Morse structures for an arbitrary open book, this  section addresses a special case of this open question where it may be answered fully: open books with one-holed torus pages.  In this restricted setting, we can characterise all the combinatorial Morse diagrams associated to a fixed open book.  This example is both interesting in its own right, and also suggestive of how one might approach the more general classification problem.

Henceforth, set $\Sigma$ to be a genus-one surface with one boundary component and fix a mapping class $\phi$.  Figure~\ref{fig:MorseMoves} shows a collection of moves on Morse diagrams, each of which should be viewed as a cut-and-replace operation on an annular Morse diagram.    Each picture represents a family of equivalent moves related by horizontal and vertical reflection.  

\begin{figure}
    \centering
    \includegraphics[width=\textwidth]{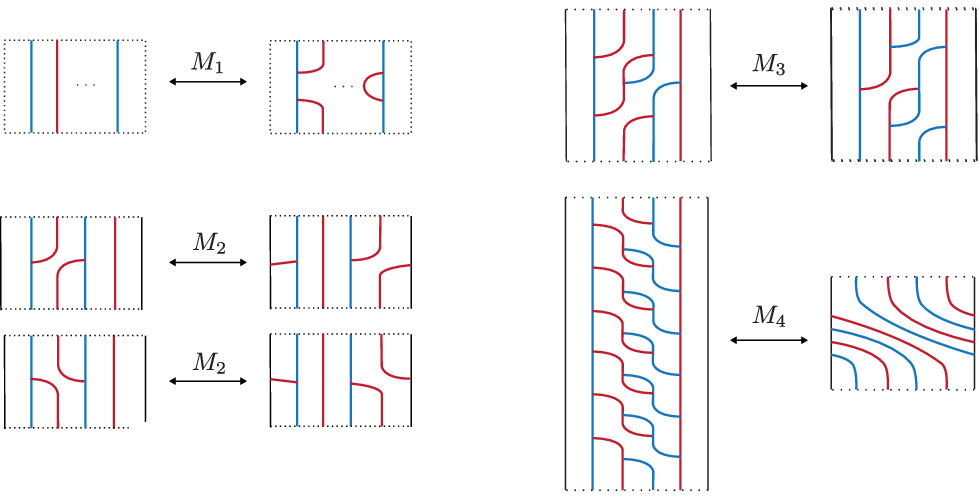}
    \caption{Morse moves for one-holed torus pages.  Pairs of diagrams may be reflected horizontally and vertically to produce other versions of the moves. }
    \label{fig:MorseMoves}
\end{figure}

\begin{theorem}\label{prop:OBDMorseEquivalence}
Two Morse diagrams represent the same open book  if and only if they can be related by a sequence of the moves from Figure~\ref{fig:MorseMoves}.
\end{theorem}

\subsection{Mapping class group of the one-holed torus}

We begin by recalling a standard presentation for the mapping class group of $\Sigma$.  Suppose that $A$ and $B$ are oriented non-separating closed curves  with algebraic intersection $i(A,B) = +1$, and let $C$ be a boundary-parallel closed curve, oriented as $\partial \Sigma$.  For any simple closed curve $X$, let $\tau_{X}$ be a right-handed Dehn twist around the curve $X$.
Then 
\[ \text{Mod}(\Sigma)\cong\langle \tau_A, \tau_B, \tau_C\ | \ \tau_{A}\tau_{B}\tau_{A} = \tau_{B}\tau_{A}\tau_{B}, (\tau_{A}\tau_{B})^{6} = \tau_{C}\rangle.\]

See, for example, Section 3.6.4 in \cite{FM}.

Any $A, B$ as above can be extended to a handle structure $\H_0=(A, B, A^*, B^*)$.

There are eight possible configurations of labeled points on the basepointed $\partial \Sigma$, and these are distinguished by the labels of the two points following the basepoint. Given an initial boundary configuration, there are six possible arc slides and two rotations.   Each rotation changes the configuration, while a distinguishing feature of the once-punctured torus is that each arc slide preserves the boundary configuration.

In any Morse diagram the boundary configurations at $t=0$ and $t=1$ agree, so we will standardise the Morse diagram, replacing the eight boundary configurations and sixty-four moves above with a set of nine moves which each preserve the initial boundary configuration. %

Our first observation is that inserting a rotation and its inverse at any $t=c$ slice clearly preserves the open book.  Second, we note that exchanging the order of a rotation and an arc slide preserves the open book, since this is simply an isotopy on the annular Morse diagram.  We use these two relations to modify a Morse diagram as follows:

The annular Morse diagram has four trace curves $A^*_\pm, B^*_\pm$ and an additional curve $C$ traced out by the basepoint. For notational simplicity and without loss of generality, suppose that the initial labeled point in the $t=0$ boundary configuration is $B^*_+$. Beginning at $t=0$ and working upwards to $t=1-\epsilon$, isotope $C$ relative to its endpoints so that after each arc slide, $C$ returns to the left of  $B^*_+$ after crossing at most one trace curve.  This process is equivalent to performing an isotopy of the Morse diagram on the annular diagram where the basepoint is not recorded, so this alteration preserves the mondoromy encoded by the Morse diagram.

We may now rewrite any given sequence of arc slides on the diagram as sequence of nine enhanced moves, each of which preserves the initial configuration.  Letting $t$ increase from $t=0$, this factorisation follows from identifying an enhanced move as the smallest sequence of original moves that preserves the configuration. The enhanced moves consist of four arc slides that preserve the configuration; a boundary parallel Dehn twist; and each of the four arc slides that don't preserve the configuration, followed by a rotation that restores it.  

\begin{figure}[h]
    \centering
    \includegraphics[width = \textwidth]{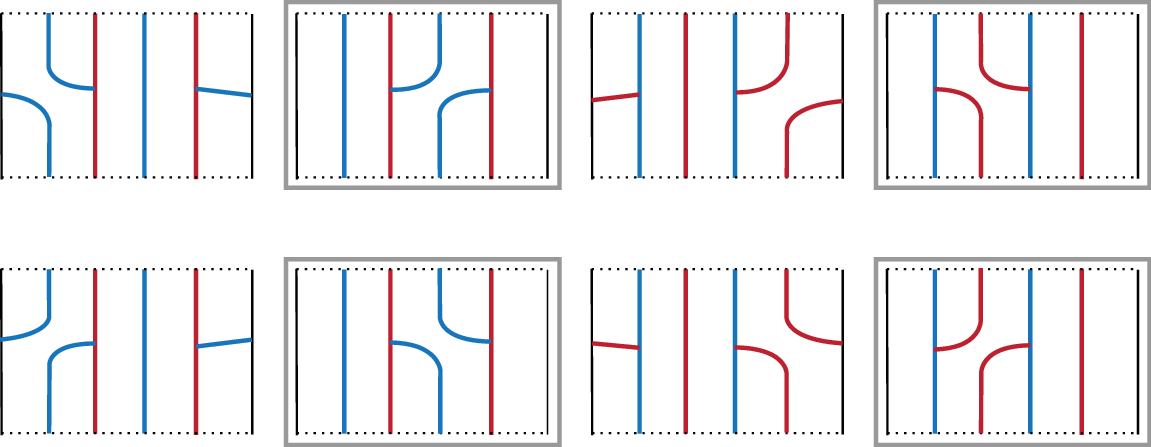}
    \caption{ Eight handle slides.  The ninth handle slide (not shown) is the full rotation that corresponds to a boundary-parallel Dehn twist. The preferred handle slides, which are introduced immediately after Corollary~\ref{cor:1}, consist of the full rotation and the four moves boxed in gray.
}
    \label{fig:diagramsforhandleslidesR}
\end{figure}

This process may produce a large number of rotations in the $t$-interval $[1-\epsilon, 1]$.  Cancelling adjacent rotations in opposite directions yields a (still potentially large) number of rotations in a single direction; to see that this number is a multiple of four, and hence a boundary-parallel twist, simply observe that the configurations at $t=1-\epsilon$ and $t=1$  agree.  

Henceforth, we will work in terms of these nine  moves, eight of which are shown in Figure~\ref{fig:diagramsforhandleslidesR}. We will refer to these as \emph{handle slides} to distinguish them from the arc slides that may not preserve the boundary configuration.   

Both arc slides and  handle slides can be viewed as functions acting on the set of handle structures associated to the fixed configuration.  An advantage of using handle slides is that each handle slide acts on the existing handle structure as an easily identified mapping class.   Conjugate mapping classes determine equivalent open books, so   we may freely choose an  initial handle structure $\H_0=(A, B, A^*, B^*)$.  Let $h(X_\circ, Y_\pm)\cdot \H_0$ denote the handle structure produced from $\mathcal{H_0}$ by performing the handle slide that teleports the trace curve labeled $X^*_\circ$ over the trace curve labeled $Y^*\pm$.

\begin{proposition}\label{lem:runoffedge}
    With the notation introduced above,  the following hold up to isotopy:
   \begin{itemize}
              \item $ h(A_+,B_+)\cdot\H_0=h(A_-,B_-)\cdot\H_0=\tau_A\H_0$
        
         \item $ h(A_-,B_+)\cdot\H_0=h(A_+,B_-)\cdot\H_0=\tau_A^{-1}\H_0$

    \item $h(B_-,A_-)\cdot\H_0=h(B_+,A_+)\cdot\H_0=\tau_B^{-1}\H_0$
    \item $ h(B_-,A_+)\cdot\H_0=h(B_+,A_-)\cdot\H_0=\tau_B\H_0$
\end{itemize}
\end{proposition}

\begin{proof}
Each of these  equalities is proven by applying the indicated handle slide to some initial handle structure and comparing the result to its image under the corresponding Dehn twist.  One case was computed in Example~\ref{ex:twist} and the others proceed similarly.  
\end{proof}

\begin{corollary}\label{cor:1} Two Morse diagrams related by an $M_2$ move represent open books with monodromy in the same mapping class. 
\end{corollary}

It follows that any handle slide except a boundary parallel twist may be replaced by a handle slide supported in the complement of the basepoint, all while preserving the associated open book.  These preferred handle slides are indicated in Figure~\ref{fig:diagramsforhandleslidesR}.

Recall that the monodromy determined by a combinatorial Morse structure $\H_0\dots \H_n$ is the unique mapping class $\phi$ such that $\phi(\mathcal{H}_{n})=\mathcal{H}_{0}$.  Letting $\phi_k$ denote the mapping class such that $\phi_k(\H_k)=\H_{k-1}$, we can factor $\phi=\phi_1\circ \phi_2\dots \circ \phi_n$. Thus, concatenating handle slides is equivalent to composing mapping classes. Proposition~\ref{lem:runoffedge} identifies the mapping class associated to a handle slide as a Dehn twist around one of the cores of the current handle structure,  but performing this handle slide produces a different set of cores that the next handle slide operate on.  As the next result shows, the correspondence between a sequence of handle slides and a factorisation of the monodromy into Dehn twists is as simple as one might hope.  

Let $h_{k}$ denote the $k^{th}$ handle slide in a Morse diagram and let $\H_n=(A_n, B_n)$ denote the handle structure associated to modifying some initial handle structure $\H_0=(A_0, B_0)$ by $h_{1}, \dots, h_{n}$. Then listing the handleslides of a combinatorial Morse structure factors the monodromy in terms of Dehn twists around cores in each successive handle structure:
\[ \phi= \tau_{x_1}^{\sigma_i} \tau_{x_2}^{\sigma_2}\dots \tau_{x_n}^{\sigma_n}\tau_C^d, \ \ x_i\in\{A_{i-1}, B_{i-1}\}.\]

Here, $C$ is a boundary parallel curve.

\begin{theorem}\label{thm:factorisation} With $x_i\in \{A_{i-1}, B_{i-1}\}$ as above,  suppose the monodromy of a Morse structure is given by $\phi= \tau_{x_1}^{\sigma_1} \tau_{x_2}^{\sigma_2}\dots \tau_{x_m}^{\sigma_m}\tau_C^d$.
For each $i$, let $k_i$ be the element of  $\{A_0, B_0\}$ given by replacing $A_i\mapsto A_0$ and $B_i\mapsto B_0$.  Then $\phi=\tau^{\sigma_m}_{k_{m}}\tau^{\sigma_{m-1}}_{k_{m-1}}\cdots\tau^{\sigma_1}_{k_{1}}\tau_C^{d}$.
\end{theorem}

\begin{proof}
    Suppose first that $d=0$.

    We will prove by induction on $m$ that $\tau_{x_1}^{\sigma_1} \tau_{x_2}^{\sigma_2}\dots \tau_{x_m}^{\sigma_m}=\tau^{\sigma_m}_{k_{m}}\tau^{\sigma_{m-1}}_{k_{m-1}}\cdots\tau^{\sigma_1}_{k_{1}}$.

For the base case, the claim $\tau_{x_1}^{\sigma_1}=\tau_{k_1}^{\sigma_1}$ follows immediately from the fact that $x_1=k_1$.

For the inductive step, suppose that the claim holds for any factorisation with at most $n$ factors and consider the mapping class $\phi_{n+1}=\tau_{x_1}^{\sigma_1}\dots\tau_{x_{n+1}}^{\sigma_{n+1}}$.

Applying the inductive hypothesis, we may rewrite this as

\[\phi_{n+1}=\tau_{k_n}^{\sigma_n}\dots\tau_1^{\sigma_1}\tau_{x_{n+1}}^{\sigma_{n+1}}.\]

Recall that for any $f \in \text{Mod}(\Sigma)$ and any simple closed curve $X$, the identity $\tau_{f(X)}=fT_Xf^{-1}$ holds.  (See, for example, Fact 3.7 in \cite{FM}.) We use this fact to rewrite $\tau_{x_{n+1}}$ as a twist around a curve in $\{A_0, B_0\}$.  Specifically, let 
\[x_{n+1}=\phi_n^{-1}k_{n+1}\phi_n.\] Again applying the inductive hypothesis, we are free to express $\phi_n$ as a product of twists around curves $k_i\in \{A_0, B_0\}$:

\begin{align}\phi_{n+1}&=\tau_{k_n}^{\sigma_n}\dots\tau_1^{\sigma_1}\big(\phi^{-1}\tau_{k_{n+1}}^{\sigma_{n+1}}\phi\big)\\ 
&=\tau_{k_n}^{\sigma_n}\dots\tau_1^{\sigma_1}\big(\tau_{k_1}^{-\sigma_1}\dots \tau_{k_n}^{-\sigma_n}\tau_{k_{n+1}}^{\sigma_{n+1}}\tau_{k_n}^{\sigma_n}\dots \tau_{k_1}^{\sigma_1}\big)\\ 
&=\tau_{k_{n+1}}^{\sigma_{n+1}}\tau_{k_n}^{\sigma_n}\dots \tau_{k_1}^{\sigma_1}.
\end{align}

The extension to $d\neq 0$ is immediate.

 \end{proof}

 As a consequence of the theorem, we may translate easily between the factorisation of $\phi$ read off from the Morse diagram and a factorisation in terms of standard generators of $\text{Mod}(\Sigma)$.

\begin{example}
Let $(\Sigma,\phi)$ be an abstract open book where $\Sigma$ is a torus with one boundary component, and $\phi = \tau_{B}^{-1}\tau_{A}\tau_{B}\tau_{A}^{-1}$. To find a Morse diagram for $(\Sigma,\phi)$, we want to find a sequence of handleslides realising $\phi^{-1} = \tau_{A}\tau^{-1}_{B}\tau^{-1}_{A}\tau_{B}$. By the proof of Theorem \ref{thm:factorisation}, we can rewrite $\phi^{-1}$ as a composition of Dehn twists about curves that are not necessarily $A$ or $B$, $\phi^{-1} = \tau_{B_{4}}\tau^{-1}_{A_{3}}\tau^{-1}_{B_{2}}\tau_{A_{1}}$, where 
\begin{itemize}
    \item $\tau_{A_{1}}$ is a positive Dehn twist about the curve $A_{1} = A$, and is isotopic to a handleslide of $B$ over $A$, and the action of this handleslide on the chosen generators $A$ and $B$ of $H_{1}(\Sigma)$ is \[(A,B)\mapsto(A,B+A)\]
    \item $\tau^{-1}_{B_{2}}$ is a Dehn twist about the curve $B_{2} = B+A$, and is isotopic to a handleslide of $\tau_{A_{1}}(A) = A$ over $\tau_{A_{1}}(B) = B+A$, \[(A,B+A)\mapsto (2A+B, A+B)\]
    \item $\tau^{-1}_{A_{3}}$ is a negative Dehn twist about the curve $A_{3} = 2A+B$, and is isotopic is a handleslide of $\tau^{-1}_{B_{2}}\tau_{A_{1}}(B) = A+B$ over $-\tau^{-1}_{B_{2}}\tau_{A_{1}}(A) = -2A-B$, \[(2A+B,A+B)\mapsto(2A+B,-A)\]
    \item $\tau_{B_{4}}$ is a positive Dehn twist about the curve $B_{4} = -A$, and is isotopic to a handleslide of $\tau^{-1}_{A_{3}}\tau^{-1}_{B_{2}}\tau_{A_{1}}(A) = 2A+B$ over $-\tau^{-1}_{A_{3}}\tau^{-1}_{B_{2}}\tau_{A_{1}}(B) = A$, \[(2A+B,-A)\mapsto(3A+B, -A).\]
\end{itemize}
See Figure~\ref{fig:ExampleMorseDiagram}.

\begin{center}
    \begin{figure}[h]
        \centering
        \includegraphics[width = 5cm]{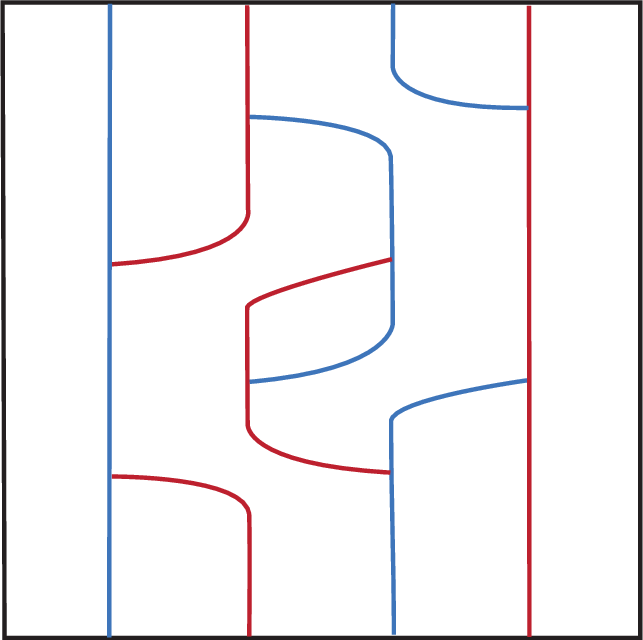}
        \caption{A Morse diagram for the open book $(\Sigma,\tau^{-1}_{B}\tau_{A}\tau_{B}\tau_{A}^{-1})$.}
        \label{fig:ExampleMorseDiagram}
    \end{figure}
\end{center}

\end{example}

\begin{example}\label{ex:phs}
This correspondence allows us to construct examples of Morse diagrams for familiar 3-manifolds. For example, the Poincar\'{e} homology sphere with reversed orientation, $\overline{PHS^{3}}$, arises as $+1$ surgery on the right-handed trefoil $T^{+}$.  It follows that  $\overline{PHS^{3}}$ admits an open book built from the fibered $T^{+}$ in $S^{3}$  by adding a negative boundary parallel Dehn twist to the monodromy.   Let $A$ and $B$ be homology generators  for the punctured-torus fibre of the $S^3$ open book with binding $T^+$.  The monodromy in this case is $\tau_{A}\tau_{B}$, so an abstract open book for $\overline{PHS^{3}}$ is $(\Sigma,\tau_{A}\tau_{B}\tau_{C}^{-1})$. 

Theorem~\ref{thm:factorisation} then implies that a Morse diagram for this monodromy can be constructed as in right-hand picture in Figure~\ref{fig:PHSexample}.
 It is well-known that the contact structure compatible with this open book decomposition is overtwisted, but this may be deduced from Theorem \ref{thm:OT}: apply the $M4$ move to eliminate the boundary parallel Dehn and apply the $M1$ move twice to produce a diagram with a left-veering handle.
\end{example}

\bibliographystyle{alpha}
\bibliography{morse}

\end{document}